\documentclass[12pt]{article}
\usepackage{mathrsfs}
\usepackage{amsfonts}
\usepackage{amsthm,amsmath,amssymb,anysize}
\newtheorem{lemma}{Lemma}[section]
\newtheorem{theorem}[lemma]{Theorem}

\newtheorem{definition}[lemma]{Definition}

\newtheorem{example}[lemma]{Example}
\usepackage[numbers,sort&compress]{natbib}
\usepackage{amsmath}
\allowdisplaybreaks[3]
\marginsize{3cm}{3cm}{3.6cm}{3cm}
\numberwithin{equation}{section}

\newcommand{\upcite}[1]{\textsuperscript{\textnormal{\cite{#1}}}}
\title{\bf{Local (Anti-)Superderivations on Finite-dimensional Nilpotent Lie Superalgebras}}
\author{{Xiaohui Chi, Huiyi Zhang, Lingxin Meng and Liming Tang\footnote{corresponding author}}\\
\small{School of Mathematical Sciences}\\
\small{Harbin Normal University}\\
\small{150025 Harbin, China}\\
\small{E-mail: limingtang@hrbnu.edu.cn}}
\date{ }
\date{ }
\begin{document}
\maketitle
 \begin{quotation}
{\small\noindent \textbf{Abstract}:
In this paper, we study local (anti-)superderivations of finite-dimensional nilpotent Lie
superalgebras over a field $\mathbb F$ of
characteristic not 2. First, we prove that every
finite-dimensional 2-step nilpotent Lie superalgebra admits
pure local (anti-)superderivations (namely, local (anti-)superderivations that are not (anti-)superderivations).
For finite-dimensional nilpotent Lie superalgebras of nilpotency index greater than
2, we obtain sufficient conditions for the existence of
pure local (anti-)superderivations.
Moreover, we prove that every 3-step nilpotent Lie
superalgebra admits a pure local superderivation.
 
 \vspace{0.05cm} \noindent{\textbf{Keywords}}: nilpotent Lie superalgebra; 
local superderivation; local anti-superderivation

\vspace{0.05cm} \noindent \textbf{Mathematics Subject Classification
2010}: 17B05, 17B40}
\end{quotation}
 \setcounter{section}{0}
\section{Introduction}
Lie superalgebras are an important class of non-associative algebras with extensive applications in theoretical physics, quantum mechanics and related fields. Nilpotent Lie superalgebras are an important subclass of Lie superalgebras, and their structure and representation theory have long been among the central research topics in algebra. The aim in this paper is to investigate local (anti-)superderivations of finite-dimensional nilpotent Lie
superalgebras over a field $\mathbb F$ of
characteristic not 2.

Derivations for an algebra 
 are core tools for characterizing its algebraic structures. Filippov introduced $\delta$-derivations of Lie algebras in \cite{97,98}, which unify a variety of important maps including derivations (for $\delta=1$) and anti-derivations (for $\delta=-1$). Subsequently, 
  Kaygorodov studied $\delta$-derivations of Jordan superalgebras
and classical Lie superalgebras in \cite{95,96}. And Kaygorodov also
investigated $\delta$-superderivations of simple finite-dimensional
Jordan superalgebras and Lie superalgebras in \cite{90}. Zusmanovich investigated $\delta$-derivations and $\delta$-superderivations of prime Lie superalgebras~in \cite{100}. 
Local derivations play an important role in the study of certain local properties of algebras. Kadison \cite{43}, Larson and Sourour \cite{17} introduced the local derivations for Banach (or associative) algebras in 1990, which turns out to be very interesting.
Researchers have extended the researches on local derivations to non-associative algebras, which has produced numerous relevant results \cite{38,3,4,5,7} . For instance, Ayupov and Kudaybergenov 
proved that every local derivation on a finite-dimensional semisimple Lie algebra over an algebraically closed field of characteristic zero is a derivation, and constructed examples of finite-dimensional nilpotent Lie algebras admitting local derivations that are not derivations in \cite{1}.  

The theory of local derivations has been naturally extended to the superalgebra setting.  Fo\v{s}ner and  Fo\v{s}ner studied local superderivations
on certain superalgebras in \cite{27}. In 2017, an early study of local superderivations on Lie superalgebras was carried out by Chen, Wang and Nan \cite{28}. 
Subsequent systematic investigations on local superderivations of simple Lie superalgebras  verified that every local superderivation for each basic classical Lie superalgebra (except for A(1,1)) over the complex field \cite{28}, $\mathfrak{q}(n)$ with $n > 3$ \cite{29}, as well as all Cartan-type Lie superalgebras \cite{31}, is a superderivation. Analogous results were further obtained for a wider class of Lie superalgebras, including the super Virasoro algebras \cite{80}, generalized quaternion algebras (regarded as Lie superalgebras) \cite{34}, the super Schr\"{o}dinger algebras \cite{37}, the N=2 super-BMS$_3$ algebra \cite{45}, and $n$-th order super Schr\"{o}dinger algebras \cite{46}. 

There exist Lie superalgebras with specific structures admitting pure local superderivations. For instance, a class of nilpotent Lie superalgebras was shown to admit pure local superderivations\cite{28}. The examples of nilpotent and solvable Lie superalgebras with pure local superderivations have been constructed in \cite{35}. 

The concept of local anti-derivations was introduced by Ayupov, Atajonov and Yusupov \cite{99}. They proved that any local anti-derivation on a solvable Lie algebra of maximal rank with filiform, Heisenberg, or abelian nilradical is an anti-derivation. And they also showed that certain solvable Lie algebras with filiform nilradical and one-dimensional complementary space admit local anti-derivations that are not anti-derivations. In \cite{150}, the authors introduced the notion of local $\delta$-superderivations. In particular, the case $\delta=-1$ corresponds to local anti-superderivations. To date, abundant results have been established for local (super)derivations, but the study of local anti-(super)derivations remains relatively limited. 

Motivated by the works in \cite{38,30}, in this paper we investigate local (anti-)superderivations of finite-dimensional nilpotent Lie superalgebras. This paper is organized as follows: In Section 2, we recall the basic notions and present local anti-superderivations as the case $\delta=-1$ of local
$\delta$-superderivations \cite{150}. We then introduce graded
local anti-superderivations and special graded local
anti-superderivations, and describe the relationships between
the corresponding spaces. In Section 3, we prove that every finite-dimensional 2-step nilpotent Lie superalgebra over a field $\mathbb{F}$ with $\operatorname{char}\mathbb{F}\neq2$ admits pure local (anti-)superderivations. In Section 4, we obtain a set of sufficient conditions for the existence of pure local (anti-)superderivations for the finite-dimensional nilpotent Lie superalgebras of nilpotency index greater than 2. Moreover, we show that 3-step nilpotent Lie superalgebras admit pure local superderivations.

Throughout this paper, we denote by $\operatorname{char}\mathbb{F}$, $\mathbb{N}^*$, $\mathbb{Z}_2$ the characteristic of the field $\mathbb{F}$, the set of positive integers, and the additive group of integers modulo $2$, respectively. Let $\mathbb{F}$ be a field with $\operatorname{char}\mathbb{F} \neq 2$, and assume that all Lie superalgebras are finite-dimensional over $\mathbb{F}$. We use $\{X \mid Y\}$ to stand for the set of all homogeneous basis elements of a superalgebra $L$, where $X$ is the collection of even basis elements and $Y$ is the collection of odd basis elements.

\section{Preliminaries}

In this section, we mainly introduce the fundamental concepts related to nilpotent Lie superalgebras.


Let $V$ be a vector space over a field $\mathbb{F}$. If $V$ admits a direct sum decomposition of subspaces $V = V_{\bar{0}} \oplus V_{\bar{1}}$, then $V$ is called a \emph{superspace} (or $\mathbb{Z}_2$-graded space).
Elements in $V_{\bar{0}}$ are called \emph{even elements,} and elements in $V_{\bar{1}}$ are called \emph{odd elements.} Both even and odd elements are referred to as \emph{$\mathbb{Z}_2$-homogeneous elements.}
For a $\mathbb{Z}_2$-homogeneous element $x$, we write $|x|$ for its $\mathbb{Z}_2$-degree. 
  If $L$ satisfies $L = L_{\bar{0}} \oplus L_{\bar{1}}$ and $L_{\alpha} L_{\beta} \subseteq L_{\alpha+\beta}$ for all $\alpha, \beta \in \mathbb{Z}_2$, then $L$ is called \emph{a superalgebra} over $\mathbb{F}$.


\begin{definition}\upcite{39}
Let $L = L_{\bar{0}} \oplus L_{\bar{1}}$ be a Lie superalgebra over a field $\mathbb{F}$. If its bracket operation $[\cdot, \cdot]$ satisfies the following two identities for all $x \in L_{\alpha}$, $y \in L_{\beta}$, $z \in L$ with $\alpha, \beta \in \mathbb{Z}_2$:
\begin{align*}
[x, y] &= -(-1)^{|x||y|}[y, x], \\
[x, [y, z]] &= [[x, y], z] + (-1)^{|x||y|}[y, [x, z]],
\end{align*}
then $L$ is called a Lie superalgebra over $\mathbb{F}$.
\end{definition}

Let $L$ be a Lie superalgebra over a field $\mathbb{F}$. The sequence
\begin{align*}
L^0 = L,\quad L^1 = [L, L],\quad L^2 = [L, L^1],\quad \cdots,\quad L^k = [L, L^{k-1}],\quad k \in \mathbb{N}^*
\end{align*}
is called \emph{the lower central series of $L$}.
Let $L$ be a Lie superalgebra over a field $\mathbb{F}$. If there exists a positive integer $n$ such that $L^n = 0$, then $L$ is called a \emph{nilpotent Lie superalgebra}. Furthermore, the smallest positive integer $n$ satisfying $L^n = 0$ is called the \emph{nilpotency index} of $L$. A Lie superalgebra with nilpotency index $n$ is called an \emph{n-step nilpotent Lie superalgebra}.
Let $L$ be a Lie superalgebra over a field $\mathbb{F}$. The set
$$
Z(L) = \{ x \in L \mid [x, y] = 0 \text{ for all } y \in L \}
$$
is called the \textit{center} of $L$.

The definition of $\delta$-superderivation for a superalgebra is given in \cite{90}. The analogous one for a Lie superalgebra is as follows: 
\begin{definition}
\ Let $L = L_{\bar{0}} \oplus L_{\bar{1}}$ be a Lie superalgebra. Fix $\delta \in \mathbb{F}$ and $D \in \mathfrak{gl}_{\alpha}(L)$, $\alpha\in \mathbb{Z}_2$.If
\[
D([x,y]) = \delta\big([D(x),y] + (-1)^{|D||x|}[x,D(y)]\big)
\] for all homogeneous $x,y$, then $D$ is called an \textit{$\alpha$-degree $\delta$-superderivation} of the Lie superalgebra $L$.
\end{definition}

If $\alpha = \bar{0}$, $D$ is said to be an \textit{even $\delta$-superderivation}; if $\alpha = \bar{1}$, $D$ is said to be an \textit{odd $\delta$-superderivation}.

Denote by $\operatorname{Der}_{\delta}(L)$ the set of all $\delta$-superderivations of $L$. Write $\operatorname{Der}_{\delta,\bar{0}}(L)$ for the subset consisting of all even $\delta$-superderivations, and $\operatorname{Der}_{\delta,\bar{1}}(L)$ for the subset consisting of all odd $\delta$-superderivations. We have
\[
\operatorname{Der}_{\delta}(L) = \operatorname{Der}_{\delta,\bar{0}}(L) \oplus \operatorname{Der}_{\delta,\bar{1}}(L).
\]

When $\delta = 1$, we have
\[
D([x,y]) = [D(x),y] + (-1)^{|D||x|}[x,D(y)],
\]
and in this case $D$ is called \emph{a superderivation} 

When $\delta = -1$, we have
\[
D([x,y]) = -[D(x),y] - (-1)^{|D||x|}[x,D(y)],
\]
and in this case $D$ is called \emph{an anti-superderivation}.
We denote
\[
\mathrm{AntiDer}_\alpha(L):=\mathrm{Der}_{-1,\alpha}(L).
\]
A $\delta$-superderivation $D$ of $L$ is called \emph{trivial} if $D(x) = 0$ for all $x \in L$.

For all $x \in L$, the map $\operatorname{ad}x$ on $L$ defined as $\operatorname{ad}x(y) = [x,y]$, $y \in L$ is a superderivation and superderivations of this form are called \emph{inner superderivations}.
\begin{definition}\upcite{30}
Let $L$ be a Lie superalgebra, and let $\Delta$ be a linear transformation on $L$.
If for every $x \in L$, there exists a superderivation $D_x$ of $L$ (depending on $x$) such that $\Delta(x) = D_x(x)$, then $\Delta$ is called a local superderivation of $L$.
\end{definition}

 
We denote by $\operatorname{LDer}(L)$ the set of all local superderivations of $L$.

 The set of all the local superderivations of $L$ of parity $\alpha$ is denoted by $\operatorname{GLDer}_{\alpha}(L)$, that is,
\[
\operatorname{GLDer}_{\alpha}(L) = \operatorname{LDer}(L) \cap \mathfrak{gl}_{\alpha}(L).
\]
Write $\operatorname{GLDer}(L) := \operatorname{GLDer}_{\bar{0}}(L) \oplus \operatorname{GLDer}_{\bar{1}}(L)$. The elements of $\operatorname{GLDer}(L)$ are called \emph{graded local superderivations} of $L$. 

\begin{definition}\upcite{30}
A linear transformation $\phi$ of parity $\alpha$ of a Lie superalgebra $L$ is called a special graded local superderivation of parity $\alpha \in \mathbb{Z}_2$ if for any $x \in L$ there exists $D^x \in \operatorname{Der}_{\alpha}(L)$ (depending on $x$) such that $\phi(x) = D^x(x)$.
\end{definition}

Denote by $\operatorname{SGLDer}_{\alpha}(L)$ the set of all the special graded local superderivations of parity $\alpha$ of $L$. Write $\operatorname{SGLDer}(L) := \operatorname{SGLDer}_{\bar{0}}(L) \oplus \operatorname{SGLDer}_{\bar{1}}(L)$. The elements of $\operatorname{SGLDer}(L)$ are called  \emph{special graded local superderivations} of $L$. $\operatorname{SGLDer}(L)$ is a subsuperspace of $\operatorname{GLDer}(L)$.We have a sequence of subspaces:
\[
\operatorname{Der}(L) \subset \operatorname{SGLDer}(L) \subset \operatorname{GLDer}(L) \subset \operatorname{LDer}(L).
\]

In the following, we introduce the local anti-superderivations on a Lie superalgebra.
\begin{definition}\upcite{150}
Let $L$ be a Lie superalgebra, and let $\Delta$ be a linear transformation on $L$.
If for every $x \in L$, there exists an anti-superderivation $T_x$ of $L$ (depending on $x$) such that $\Delta(x) = T_x(x)$, then $\Delta$ is called a local anti-superderivation of $L$.
\end{definition}
We denote by $\operatorname{LAntiDer}(L)$ the set of all local anti-superderivations of $L$. It is easy to see that $\operatorname{LAntiDer}(L)$ forms a vector subspace of End$_{\mathbb{F}}(L)$. 

 The set of all the local anti-superderivations of $L$ of parity $\alpha$ is denoted by $\operatorname{GLAntiDer}_{\alpha}(L)$, 
\[
\operatorname{GLAntiDer}_{\alpha}(L) = \operatorname{LAntiDer}(L) \cap \mathfrak{gl}_{\alpha}(L).
\]
$\operatorname{GLAntiDer}(L) := \operatorname{GLAntiDer}_{\bar{0}}(L) \oplus \operatorname{GLAntiDer}_{\bar{1}}(L)$. The elements of $\operatorname{GLAntiDer}(L)$ are called \emph{graded local anti-superderivations} of $L$. $\operatorname{GLAntiDer}(L)$ is a subsuperspace of ${\rm{End}}_{\mathbb{F}}(L)$.

\begin{definition}
A linear transformation $\Delta$ of parity $\alpha$ of a Lie superalgebra $L$ is called a special graded local anti-superderivation of parity $\alpha \in \mathbb{Z}_2$ if for any $x \in L$ there exists $T^x \in \operatorname{AntiDer}_{\alpha}(L)$ (depending on $x$) such that $\Delta(x) = T^x(x)$.
\end{definition}

Denote by $\operatorname{SGLAntiDer}_{\alpha}(L)$ the set of all the special graded local anti-superderivations of parity $\alpha$ of $L$. Write $\operatorname{SGLAntiDer}(L) := \operatorname{SGLAntiDer}_{\bar{0}}(L) \oplus \operatorname{SGLAntiDer}_{\bar{1}}(L)$. The elements of $\operatorname{SGLAntiDer}(L)$ are called \emph{special graded local anti-superderivations} of $L$. We have a sequence of subspaces:
\[
\operatorname{AntiDer}(L) \subset \operatorname{SGLAntiDer}(L) \subset \operatorname{GLAntiDer}(L) \subset \operatorname{LAntiDer}(L).
\]

We notice that $\mathrm{GLAntiDer}_\alpha$ only requires that $\Delta$ is homogeneous, whereas $\mathrm{SGLAntiDer}_\alpha$ additionally demands that every witness belongs to $\mathrm{AntiDer}_\alpha$.


\section{2-Step Nilpotent Lie Superalgebras}

Let $L$ be a finite-dimensional 2-step nilpotent Lie superalgebra over a field $\mathbb{F}$ with $\operatorname{char}\mathbb{F}\neq 2$, satisfying $L^1=[L,L]$ and $L^2=0$.

We choose a homogeneous complement of $L^1$ in $L$, such that $L$ is the direct sum of this homogeneous complement and $L^1$.
Take $B' = \{e_{l+1},\dots,e_m \mid e_{m+t+1},\dots,e_{m+n}\}$ as a homogeneous basis of $L^1$.
By the basis extension theorem, we extend $B'$ to a homogeneous basis of $L$:
\[
B = \{e_1,\dots,e_m \mid e_{m+1},\dots,e_{m+n}\},
\]
where
\[
L^1=\operatorname{span}\big\{\,[e_i,e_j]\,\big|\ i,j\in\{1,\dots,l\}\cup\{m+1,\dots,m+t\}\,\big\}.
\]
All results in this section are established under the above setting.

\subsection{Local Superderivations}

\begin{lemma}\label{B}
Let $L$ be a 2-step nilpotent Lie superalgebra over a field $\mathbb{F}$ with $\operatorname{char}\mathbb{F} \neq 2$.
For any $\lambda \in \mathbb{F}$, define a linear transformation $D_\lambda$ on $L$ by
\begin{align*}
D_\lambda(e_i) =
\begin{cases}
\lambda e_i, & e_i \in L \setminus L^1, \\
2\lambda e_i, & e_i \in L^1.
\end{cases}
\end{align*}
Then $D_\lambda$ is a superderivation of $L$.
\end{lemma}

\begin{proof}
 $D_\lambda$ is an even linear transformation on $L$.
For arbitrary elements
\begin{align*}
x = \sum_{i} \alpha_i e_i, \quad y = \sum_{j} \beta_j e_j, \quad \alpha_i, \beta_j \in \mathbb{F},
\end{align*}
It is easy to see the following:
if $e_i, e_j \in L \setminus L^1$, then $D_\lambda([e_i, e_j]) = 2\lambda [e_i, e_j]$; 
if at least one of $e_i, e_j \in L^1$, then $D_\lambda([e_i, e_j]) = 0$.

Therefore,
\begin{align*}
D_\lambda([x, y]) &= D_\lambda\left( \sum_{i,j} \alpha_i \beta_j [e_i, e_j] \right) = 2\lambda \sum_{i,j} \alpha_i \beta_j [e_i, e_j] \\
&= \sum_{i,j} \alpha_i \beta_j [\lambda e_i, e_j] + \sum_{i,j} \alpha_i \beta_j [e_i, \lambda e_j] \\
&= \sum_{i,j} \alpha_i \beta_j [D_\lambda(e_i), e_j] + \sum_{i,j} \alpha_i \beta_j [e_i, D_\lambda(e_j)] \\
&= [D_\lambda(x), y] + [x, D_\lambda(y)]
\end{align*}
Hence $D_\lambda$ is an even superderivation of $L$.
\end{proof}

\begin{theorem}\label{A}
Let $L$ be a 2-step nilpotent Lie superalgebra over a field $\mathbb{F}$ with $\operatorname{char}\mathbb{F} \neq 2$.
Then $L$ admits pure local superderivations.
\end{theorem}

\begin{proof}
Define a homogeneous linear map $\Delta: L \to L$ by
\begin{align*}
\Delta(e_i) =
\begin{cases}
0, & e_i \in L \setminus L^1, \\
2e_i, & e_i \in L^1.
\end{cases}
\end{align*}

We first show that $\Delta$ is not a superderivation.
 There exist homogeneous elements $p, q \in L \setminus L^1$ such that $0\neq[p, q]\in L^1$.
A direct computation shows that
\begin{align*}
&\Delta([p, q]) = 2[p, q] \neq 0, \\
&[\Delta(p), q] + (-1)^{|\Delta||p|}[p, \Delta(q)] = [0, q] + [p, 0] = 0.
\end{align*}
Thus $\Delta$ fails the superderivation identity.

We now prove that $\Delta$ is a local superderivation.
For any element $x = \sum\limits_{i=1}^{m+n} \lambda_i e_i \in L$, we have
\begin{align*}
\Delta(x) = 2\sum\limits_{i=l+1}^{m} \lambda_i e_i + 2\sum\limits_{i=m+t+1}^{m+n} \lambda_i e_i.
\end{align*}
We consider three cases:

\medskip
\noindent\textit{Case 1.}
Suppose that there exists an even basis element outside $L^1$ with nonzero coefficient. That is,
\[
\lambda_j \neq 0,\qquad 1 \leq j \leq l.
\]

Define a linear map $D_x: L \to L$ with its even component
\[
(D_x)_{\bar{0}}(e_k) =
\begin{cases}
\dfrac{2}{\lambda_j} \displaystyle\sum_{i=l+1}^{m} \lambda_i e_i, & k=j \\[12pt]
0, & k\neq j
\end{cases}
\]
and odd component
\[
(D_x)_{\bar{1}}(e_k) =
\begin{cases}
\dfrac{2}{\lambda_j} \displaystyle\sum_{i=m+t+1}^{m+n} \lambda_i e_i, &  k=j \\[12pt]
0, & k\neq j
\end{cases}
\]
Equivalently, the full map can be written as
\[
D_x(e_k) =
\begin{cases}
\dfrac{2}{\lambda_j} \left( \displaystyle\sum_{i=l+1}^{m} \lambda_i e_i + \displaystyle\sum_{i=m+t+1}^{m+n} \lambda_i e_i \right), &  k=j \\[12pt]
0, & k\neq j
\end{cases}
\]

Since $(D_x)_{\bar{0}}(L) \subseteq (L^1)_{\bar{0}} \subseteq Z(L)$, $(D_x)_{\bar{1}}(L) \subseteq (L^1)_{\bar{1}} \subseteq Z(L)$, so
\[
(D_x)_{\bar{0}}([a,b]) =0= [(D_x)_{\bar{0}}(a),b] + [a,(D_x)_{\bar{0}}(b)].
\]
\[
(D_x)_{\bar{1}}([a,b]) =0= [(D_x)_{\bar{1}}(a),b] + (-1)^{|a|}[a,(D_x)_{\bar{1}}(b)].
\]
for any homogeneous elements $a,b\in L$.
\[
D_x = (D_x)_{\bar{0}} + (D_x)_{\bar{1}},
\]
where
\[
(D_x)_{\bar{0}} \in Der_{\bar{0}}(L),\qquad
(D_x)_{\bar{1}} \in Der_{\bar{1}}(L).
\]

Therefore, $D_x$ is a superderivation.
Then
\[
\begin{aligned}
D_x(x)&=\lambda_j \frac{2}{\lambda_j}\left(\sum_{i=l+1}^{m}\lambda_i e_i+\sum_{i=m+t+1}^{m+n}\lambda_i e_i\right)\\
&=2\left(\sum_{i=l+1}^{m}\lambda_i e_i+\sum_{i=m+t+1}^{m+n}\lambda_i e_i\right)=\Delta(x).
\end{aligned}
\]

\medskip
\noindent\textit{Case 2.}
Assume that the even component of $x$ outside $L^1$ is zero, while there exists an odd basis of $x$ outside $L^1$ is nonzero. That is,
\[
\lambda_i = 0,\qquad 1\leq i\leq l,
\]
and there exists some
\[
j\in \{m+1,\dots,m+t\}
\]
such that $\lambda_j \neq 0.$

Define a linear map $D_x: L \to L$ with its even component
\[
(D_x)_{\bar{0}}(e_k) =
\begin{cases}
\dfrac{2}{\lambda_j} \displaystyle\sum_{i=m+t+1}^{m+n} \lambda_i e_i, & k=j  \\[12pt]
0, & k\neq j
\end{cases}
\]
and its odd component
\[
(D_x)_{\bar{1}}(e_k) =
\begin{cases}
\dfrac{2}{\lambda_j} \displaystyle\sum_{i=l+1}^{m} \lambda_i e_i, & k=j \\[12pt]
0, & k\neq j
\end{cases}
\]
Equivalently, the full map can be written as
\[
D_x(e_k) =
\begin{cases}
\dfrac{2}{\lambda_j} \left( \displaystyle\sum_{i=l+1}^{m} \lambda_i e_i + \displaystyle\sum_{i=m+t+1}^{m+n} \lambda_i e_i \right), &k=j \\[12pt]
0, &k\neq j
\end{cases}
\]
By a similar argument, $D_x$ is a superderivation and satisfies $D_x(x) = \Delta(x)$.

\medskip

\medskip
\noindent\textit{Case 3.}
Suppose $\lambda_i = 0$ for all $i \in \{1, \dots, l\} \cup \{m+1, \dots, m+t\}$. Then
\begin{align*}
\Delta(x) = 2\sum\limits_{i=l+1}^{m} \lambda_i e_i + 2\sum\limits_{i=m+t+1}^{m+n} \lambda_i e_i.
\end{align*}
Take the superderivation $D_x = D_1$ given in Lemma \ref{B}, i.e.,
\begin{align*}
D_x(e_i) =
\begin{cases}
e_i, & e_i \in L \setminus L^1, \\
2e_i, & e_i \in L^1.
\end{cases}
\end{align*}
Then we have $D_x(x) = \Delta(x)$.

\medskip

In all cases, for every $x \in L$, there exists a superderivation $D_x$ such that $\Delta(x) = D_x(x)$.
Therefore, $\Delta$ is an even graded local superderivation but not an even superderivation, i.e.,
\[
\operatorname{Der}_{\bar{0}}(L) \subsetneq \operatorname{GLDer}_{\bar{0}}(L)
\]
Hence, $\operatorname{Der}(L) \subsetneq \operatorname{LDer}(L)$,
which shows that $L$ admits a pure local superderivation.
\end{proof}

\begin{example}\label{F}
Let $L$ be a finite-dimensional even-centered  Heisenberg Lie superalgebra over a field $\mathbb{F}$ with $\operatorname{char}\mathbb{F} \neq 2$. Let
\begin{align*}
B = \{e_1, \dots, e_m, e_{m+1}, \dots, e_{2m}, e_0 \mid e_{-1}, \dots, e_{-n}\}
\end{align*}
be a basis of $L$, with nonzero bracket relations:
\begin{align*}
[e_i, e_{m+i}] = e_0, \quad [e_{-t}, e_{-t}] = e_0, \quad i \in \{1, \dots, m\},\ t \in \{1, \dots, n\}.
\end{align*}
Note that $L^1 = \operatorname{span}\{e_0\}$, so $L$ is indeed a 2-step nilpotent Lie superalgebra.

Define a homogeneous linear map $\Delta: L \to L$ by
\begin{align*}
\Delta(e_k) =
\begin{cases}
e_0, & k = 0, \\
0, & k \neq 0.
\end{cases}
\end{align*}
If $m\geq 1$, then
\[
\Delta([e_1, e_{m+1}]) = \Delta(e_0) = e_0\neq 0=[\Delta(e_1), e_{m+1}] + (-1)^{|\Delta||e_1|}[e_1, \Delta(e_{m+1})] 
\]
If $m=0$, $n\geq 1$.We have
\[
\Delta([e_{-1},e_{-1}])=\Delta(e_0)=e_0\neq 0=[\Delta(e_{-1}),e_{-1}]
+[e_{-1},\Delta(e_{-1})]
\]

Thus $\Delta$ is not a superderivation.
which implies that $\Delta$ is not an even superderivation. We now prove that $\Delta$ is an even graded local superderivation.

For any element $x = \displaystyle\sum_{i=-n}^{2m} \lambda_i e_i$, we have $\Delta(x) = \lambda_0 e_0$. We consider two cases:

\medskip
\noindent\textit{Case 1.}
Suppose that there exists an index
\[
j\in\{-n,\dots,-1,1,\dots,2m\}
\]
such that
\[
\lambda_j \neq 0.
\]

Define a linear map $D_x : L \to L$ by
\[
D_x(e_j) = \frac{\lambda_0}{\lambda_j} e_0,
\]
and
\[
D_x(e_i)=0,\qquad i\neq j.
\]

It is clear that $D_x$ is a superderivation, and
\[
D_x(x) = \lambda_0 e_0 = \Delta(x).
\]

\medskip
\noindent\textit{Case 2.}
Suppose $\lambda_i = 0$ for all $i \neq 0$. Define a linear map $D_x$ by
\[
D_x(e_i) =
\begin{cases}
e_0, & i = 0, \\
\frac{1}{2}e_i, & i \neq 0.
\end{cases}
\]
By Lemma \ref{B}, $D_x$ is a superderivation, and
\[
D_x(x) = \lambda_0 D_x(e_0) = \lambda_0 e_0 = \Delta(x).
\]

\medskip
Therefore, $\Delta$ is an even graded local superderivation. Furthermore, $L$ admits pure local superderivations.

\end{example}

\subsection{Local Anti-superderivations}

\begin{lemma}\label{Q}
Let $L$ be a 2-step nilpotent Lie superalgebra over a field $\mathbb{F}$ with $\operatorname{char}\mathbb{F}\neq 2$. Define a linear map $T_\lambda: L \to L$ by
\[
T_\lambda(e_i) =
\begin{cases}
\lambda e_i, & e_i \in L\setminus L^1, \\
-2\lambda e_i, & e_i \in L^1.
\end{cases}
\]
Then $T_\lambda$ is an anti-superderivation of $L$.
\end{lemma}

\begin{proof}

It is clear that $T_\lambda$ is an even linear transformation on $L$. For arbitrary elements $x = \displaystyle\sum_{i} \alpha_i e_i$, $y =\displaystyle\sum_{j} \beta_j e_j \in L$, we have
\[
\begin{aligned}
T_\lambda([x,y]) &= T_\lambda\left( \sum_{i,j} \alpha_i\beta_j [e_i,e_j] \right)= \sum_{i,j} \alpha_i\beta_j T_\lambda([e_i,e_j])=-2\lambda \sum_{i,j} \alpha_i\beta_j [e_i,e_j].
\end{aligned}
\]
On the other hand,
\[
\begin{aligned}
&-[T_\lambda(x), y] - [x, T_\lambda(y)]\\
&= -\left[ \sum_i \alpha_i T_\lambda(e_i), \sum_j \beta_j e_j \right] -\left[ \sum_i \alpha_i e_i, \sum_j \beta_j T_\lambda(e_j) \right] \\
&= -\left[ \sum_i \alpha_i \lambda e_i, \sum_j \beta_j e_j \right]-\left[ \sum_i \alpha_i e_i, \sum_j \beta_j \lambda e_j \right] \\
&= -\lambda \sum_{i,j} \alpha_i\beta_j [e_i,e_j] - \lambda \sum_{i,j} \alpha_i\beta_j [e_i,e_j] \\
&= -2\lambda \sum_{i,j} \alpha_i\beta_j [e_i,e_j].
\end{aligned}
\]
Therefore, $T_\lambda([x,y]) = -[T_\lambda(x), y] - [x, T_\lambda(y)]$, which means $T_\lambda$ is an even anti-superderivation.
\end{proof}

\begin{theorem}
Let $L$ be a 2-step nilpotent Lie superalgebra over a field $\mathbb{F}$ with $\operatorname{char}\mathbb{F}\neq 2$. Then $L$ admits pure local anti-superderivation.
\end{theorem}

\begin{proof}
Define a homogeneous linear map $\Delta: L\rightarrow L$ by
\[
\Delta(e_i) =
\begin{cases}
0, & e_i \in L \setminus L^1, \\
-2e_i, & e_i \in L^1.
\end{cases}
\]

First, we verify that $\Delta$ is not an even anti-superderivation. Take $p,q \in L \setminus L^1$ such that $0\neq[p,q]\in L^1$. Then
\[
\Delta([p, q])=-2([p, q])\neq0=-[\Delta(p), q]-(-1)^{|\Delta||p|}[p, \Delta(q)] .
\]

Next, we prove that $\Delta$ is an even graded local anti-superderivation . Let $x =\displaystyle\sum_{i=1}^{m+n} \lambda_i e_i$. Then
\[
\Delta(x) = -2\left(\sum_{i=l+1}^m \lambda_i e_i+\sum_{i=m+t+1}^{m+n} \lambda_i e_i\right).
\]

We consider the following cases:

\medskip
\noindent\textit{Case 1.}
Suppose that there exists an even basis element outside $L^1$ with nonzero coefficient. That is,
\[
\lambda_j \neq 0,\qquad 1 \leq j \leq l.
\]
Define a linear map $T_x: L\to L$ by
 \[
 T_x(e_k) =
 \begin{cases}
 \dfrac{-2}{\lambda_j}\left(\displaystyle\sum_{i=l+1}^{m}\lambda_i e_i+\displaystyle\sum_{i=m+t+1}^{m+n}\lambda_i e_i\right),&  k=j \\[8pt]
 0, & k\neq j
 \end{cases}
 \]
It is straightforward to check that $T_x$ is an anti-superderivation, and
\[
\begin{aligned}
T_x(x)&=-\lambda_j \frac{2}{\lambda_j}\left(\sum_{i=l+1}^{m}\lambda_i e_i+\sum_{i=m+t+1}^{m+n}\lambda_i e_i\right)\\
&=-2\left(\sum_{i=l+1}^{m}\lambda_i e_i+\sum_{i=m+t+1}^{m+n}\lambda_i e_i\right)=\Delta(x).
\end{aligned}
\]

\noindent\textit{Case 2.}
Assume that the even component of $x$ outside $L^1$ is zero, while the odd component of $x$ outside $L^1$ is nonzero. That is,
\[
\lambda_i = 0,\qquad 1\leq i\leq l,
\]
and there exists some
\[
j\in \{m+1,\dots,m+t\}
\]
such that
\[
\lambda_j \neq 0.
\]
Define a linear map $T_x: L \to L$
\[
T_x(e_k) =
\begin{cases}
\dfrac{-2}{\lambda_j} \left( \displaystyle\sum_{i=l+1}^{m} \lambda_i e_i + \displaystyle\sum_{i=m+t+1}^{m+n} \lambda_i e_i \right), &k=j \\[12pt]
0, &k\neq j
\end{cases}
\]
By a similar argument, $T_x$ is an anti-superderivation and  $T_x(x) = \Delta(x)$.

\medskip
\noindent\textit{Case 3.}
Suppose $\lambda_j = 0$ for all $j\in \{1,\cdots,l\}\cup \{m+1,\cdots,m+t\}$. Then $x\in L^1$, and $\Delta(x) = -2x$.

Take the anti-superderivation $T_x=T_1$ given in Lemma 3.4, where
\begin{align*}
  T_x(e_i)=
  \begin{cases}
  e_i,\quad &e_i\in L\setminus L^1, \\
  -2e_i,\quad &e_i\in L^1.
  \end{cases}
\end{align*}
Then $T_x(x) = -2x = \Delta(x)$.

In conclusion, $\operatorname{AntiDer}_{\bar{0}}(L)\subsetneq \operatorname{GLAntiDer}_{\bar{0}}(L)$.

Furthermore, $\operatorname{AntiDer}(L)\subsetneq \operatorname{LAntiDer}(L)$. Therefore, $L$ admits pure local anti-superderivation.
\end{proof}

\begin{example}
Let $L$ be a finite-dimensional  even-centered Heisenberg Lie superalgebra over a field $\mathbb{F}$ with $\operatorname{char}\mathbb{F}\neq 2$. Let
  \begin{align*}
    B=\{e_1,\cdots, e_m,e_{m+1},\cdots,e_{2m},e_0\mid e_{-1},\cdots,e_{-n}\}
  \end{align*}
be a basis of $L$, with non-zero brackets given by
 \begin{align*}
   [e_i,e_{m+i}]=e_0,\quad [e_{-t},e_{-t}]=e_0,\quad i\in \{1,\cdots,m\},\ t\in \{1,\cdots,n\}.
  \end{align*}
Define a linear map $\Delta: L\rightarrow L$ by
  \begin{align*}
  \Delta(e_k)=
  \begin{cases}
    -e_0, &\quad k=0,\\
    0, &\quad k\neq 0.
    \end{cases}
  \end{align*}
 $\Delta$ is not an even anti-superderivation, but $\Delta$ is an even graded local anti-superderivation. Furthermore, $L$ admits pure local anti-superderivation.

\end{example}


\section{$n$-step Nilpotent Lie Superalgebras ($n\geq$ 3)}

Let $B'=\{e_{l+1},\dots,e_m\mid e_{m+t+1},\dots,e_{m+n}\}$ be a homogeneous basis of $L^1$. By the basis extension theorem, we obtain a homogeneous basis of $L$:
\[
B=\{e_1,\dots,e_m\mid e_{m+1},\dots,e_{m+n}\}.
\]
With the above homogeneous basis, we obtain the following results.

\subsection{Local Superderivations}

\begin{theorem}\label{C}
Let $L$ be a finite-dimensional $n$-step nilpotent Lie superalgebra over a field $\mathbb{F}$ with nilpotency index $n\geq 3$. If there exists a non-trivial homogeneous superderivation $D\in \operatorname{Der}(L)$ such that $D(L^1)\subseteq Z(L)$ and $ D| _{L^1}\neq 0 $, then $L$ admits pure local superderivations.
\end{theorem}

\begin{proof}
Since $D|_{L^1}\neq 0$, then there exists $e_r\in L^1$ such that $D(e_r)\neq 0$.
 As $L^1=[L,L]$, write $e_r=\displaystyle\sum_{k}[p_k,q_k]$,where $p_k,q_k$ are homogeneous elements. We have
 \[
 D(e_r)=\sum_{k}D([p_k,q_k])\neq 0,
 \]
 hence there exists $p=p_k,q=q_k$ such that $D([p,q])\neq0$.

Define a linear map $\Delta: L\rightarrow L$ by
\begin{align*}
  \Delta(e_i)=
  \begin{cases}
    D(e_i), \quad &e_i\in L^1,\\
    0,\quad &e_i\in L\setminus L^1.
  \end{cases}
\end{align*}

First, note that
\begin{align*}
  \Delta([p,q])=D([p,q])\neq0=[\Delta(p),q]+(-1)^{|\Delta||p|}[p,\Delta(q)],
\end{align*}
which implies that $\Delta$ is not a superderivation of $L$.

Next, we show that $\Delta$ is a local superderivation of $L$. For any element $x=\sum\limits_{i=1}^{m+n}\lambda_i e_i\in L$, we have
\begin{align*}
\Delta(x)=\sum_{i=l+1}^{m}\lambda_i D(e_i)+\sum_{i=m+t+1}^{m+n}\lambda_i D(e_i)\in D(L^1)\subseteq Z(L).
\end{align*}

We consider the following cases:

\medskip
\noindent\textit{Case 1.}
Suppose there exists a basis element $e_j$ outside $L^1$ with nonzero coefficient, i.e.
\[
\lambda_j\neq 0,\qquad  j\in\{1,2,\dots,l\}\cup \{m+1,\dots,m+t\}
\]

Define a linear map $D_x:L\rightarrow L$ by
  \[
 D_x(e_k) =
 \begin{cases}
 \dfrac{1}{\lambda_j}\left(\displaystyle\sum_{i=l+1}^{m}\lambda_i D(e_i)+\displaystyle\sum_{i=m+t+1}^{m+n}\lambda_i D(e_i)\right), & k=j \\[8pt]
 0, & k\neq j
 \end{cases}
 \]
It is obvious that $D_x$ is a superderivation of $L$, and
\begin{align*}
D_x(x)&=\lambda_{j} D_x(e_{j})=\Delta(x).
\end{align*}

\medskip
\noindent\textit{Case 2.}
Suppose $\lambda_i=0$ for all $i\in\{1,\cdots,l\}\cup \{m+1, \cdots, m+t\}$. Define the superderivation $D_x=D$. Then
\begin{align*}
  D_x(x)=D(x)=D\left(\sum_{i=l+1}^{m}\lambda_i e_i+\sum_{i=m+t+1}^{m+n}\lambda_i e_i\right)=\Delta(x).
\end{align*}

Therefore, $\Delta$ is a local superderivation of $L$, but not a superderivation of $L$. Hence,
$\operatorname{Der}(L)\subsetneq \operatorname{LDer}(L)$, $L$ admits pure local superderivation.
\end{proof}

\begin{theorem}
Let $L$ be a 3-step nilpotent Lie superalgebra over a field $\mathbb{F}$. Then there exists a non-trivial homogeneous superderivation $D\in \operatorname{Der}(L)$ of $L$ such that $D(L^1)\subseteq Z(L)$ and $ D| _{L^1}\neq 0 $. Furthermore, $L$ admits pure local superderivations.
\end{theorem}

\begin{proof}
Since $L$ is a 3-step nilpotent Lie superalgebra, we have $L^3=0$, $L^2\neq0$, and $L^2\subseteq Z(L)$. Then there exist homogeneous elements $x\in L$ and $y\in L^1$ such that $0\neq [x,y]\in L^2$. Taking $D=\operatorname{ad}x$, we have
\[
D(y)=\operatorname{ad}x(y)=[x,y]\in L^2\subseteq Z(L).
\]
Hence, $L$ admits pure local superderivations.
\end{proof}

\begin{theorem}
Let $L$ be a finite-dimensional $n$-step nilpotent Lie superalgebra over a field $\mathbb{F}$ with $n>3$. If $[L^{n-3},L^1]\neq 0$, then there exists a non-trivial homogeneous superderivation $D\in \operatorname{Der}(L)$ of $L$ such that $D(L^1)\subseteq Z(L)$ and $ D| _{L^1}\neq 0 $. Furthermore, $L$ admits pure local superderivations.
\end{theorem}

\begin{proof}
When $n\geq 4$, since $[L^{n-3},L^1]\neq 0$, there exist homogeneous elements $x\in L^{n-3}$ and $y\in L^1$ such that $0\neq[x,y]\in L^{n-1}\subseteq Z(L)$. Taking $D=\operatorname{ad}x$, we have
\[
D(y)=\operatorname{ad}x(y)=[x,y]\in Z(L).
\]
\end{proof}
For the algebras discussed above, the existence of an inner superderivation guarantees the existence of pure local superderivations. However, a large number of finite-dimensional nilpotent Lie superalgebras satisfy the condition $[L^{n-3},L^1]= 0$, which means  the construction in the proof of Theorem 4.3 does not apply. Nevertheless, such algebras still possess superderivations satisfying the conditions of Theorem 4.1.
We present a example below.

\begin{example}\label{Q}
Let $L$ be a finite-dimensional nilpotent Lie superalgebra over the complex field. Let
\begin{align*}
B=\{e_1,e_2,e_3,e_4,e_5\mid e_6,e_7,e_8\}
\end{align*}
be a basis of $L$, with non-zero brackets given by
\begin{flalign*}
   [e_1, e_2]&=e_3, \quad [e_1, e_3]=2e_4, \quad [e_1, e_4]=3e_5,  \quad [e_2, e_3]=e_5, \\
   [e_1, e_6]&=e_7, \quad [e_1, e_7]=e_8,  \quad [e_2, e_6]=e_8. 
\end{flalign*}
By definition, $Z(L)=\operatorname{span}\{e_5,e_8\}$, $L^4=0$, and $[L^1,L^1]=0$.

Define a superderivation $D:L\rightarrow L$ by
\begin{align*}
  D(e_2)=e_4,\quad D(e_3)=3e_5,\quad D(e_i) = 0,\quad \mbox{for all}\quad i\notin\{2,3\}.
\end{align*}
Since $L^1 = \operatorname{span}\{e_3,e_4,e_5,e_7,e_8\}$, it is clear that $D(L^1)\subseteq Z(L)$ and $ D| _{L^1}\neq 0 $, which satisfies the assumption of Theorem 4.1. Now define a linear map $\Delta:L\rightarrow L$ by
\begin{align*}
\Delta(e_3) = 3e_5,\quad \Delta(e_i) = 0,\quad \mbox{for}\quad  i\neq 3.
\end{align*}
Then $\Delta$ is a local superderivation. Moreover, since
\begin{align*}
\Delta([e_1,e_2])=\Delta(e_3)=3e_5\neq 0=[\Delta(e_1),e_2]+(-1)^{|\Delta||e_1|}[e_1,\Delta(e_2)],
\end{align*}
$\Delta$ is not a superderivation.

In conclusion, $L$ admits pure local superderivations.
\end{example}

\subsection{Local Anti-superderivations}

\begin{theorem}
Let $L$ be an $n$-step nilpotent Lie superalgebra over a field $\mathbb{F}$ with $n\geq3$. If there exists a non-trivial homogeneous anti-superderivation $T\in\operatorname{AntiDer}(L)$ such that $T(L^1)\subseteq Z(L)$ and $ T| _{L^1}\neq 0 $, then $L$ admits pure local anti-superderivations.
\end{theorem}

\begin{proof}
Since $T|_{L^1}\neq 0$, there exists $e_r\in L^1$ such that $T(e_r)\neq 0$.
 As $L^1=[L,L]$, write $e_r=\sum_{k}[p_k,q_k]$, where $p_k,q_k$ are homogeneous elements. We have
 \[
 T(e_r)=\sum_{k}T([p_k,q_k])\neq 0,
 \]
 hence there exists $p=p_k,q=q_k$ such that $T([p,q])\neq0$.

Define a linear map $\Delta:L\to L$ by
\begin{align*}
  \Delta(e_i)=
  \begin{cases}
    T(e_i), \quad &e_i\in L^1,\\
    0,\quad &e_i\in L\setminus L^1.
  \end{cases}
\end{align*}
Since
\[
-[\Delta(p),q] -(-1)^{|\Delta||p|} [p,\Delta(q)] = 0 \neq T([p,q]) = \Delta([p,q]),
\]
$\Delta$ is not an anti-superderivation.

Next, we show that $\Delta$ is a local anti-superderivation. For any element $x=\sum\limits_{i=1}^{m+n}\lambda_i e_i\in L$, we have
\begin{align*}
\Delta(x)=\sum_{i=l+1}^{m}\lambda_i T(e_i)+\sum_{i=m+t+1}^{m+n}\lambda_i T(e_i).
\end{align*}
We consider the following cases:

\medskip
\noindent\textit{Case 1.}
Suppose there exists a basis element $e_j$ outside $L^1$ with nonzero coefficient, i.e.
\[
\lambda_j\neq 0,\qquad  j\in\{1,2,\dots,l\}\cup \{m+1,\dots,m+t\}
\]

Define a linear map $T_x: L\to L$ by
 \[
 T_x(e_k) =
 \begin{cases}
\dfrac{1}{\lambda_j}\left(\displaystyle\sum_{i=l+1}^m \lambda_i T(e_i) + \displaystyle\sum_{i=m+t+1}^{m+n} \lambda_i T(e_i)\right),&   k=j \\[8pt]
 0, &k\neq j
 \end{cases}
 \]
It is clear that $T_x$ is an anti-superderivation. Moreover,
  \[
  T_x(x) =  \lambda_{j} T_x(e_j) = \Delta(x).
  \]

\medskip
\noindent\textit{Case 2.}
Suppose $\lambda_j=0$ for all $j\in  \{1,\cdots,l\}\cup \{m+1,\cdots,m+t\}$. Define the linear map $T_x = T$. Then
  \[
  T_x(x) = T(x) = T\left(\displaystyle\sum_{i=l+1}^m \lambda_i e_i + \displaystyle\sum_{i=m+t+1}^{m+n} \lambda_i e_i\right)  = \Delta(x).
  \]

Therefore, $L$ admits pure local anti-superderivations.
\end{proof}

\begin{example}
Let $L$ be a nilpotent Lie superalgebra with nilpotency index $p(p\geq3)$. Let $B = \{e_1,e_2,\dots,e_m \mid e_{m+1},\dots,e_{m+n}\}$ be a basis of $L$, where $e_1,e_{m+1}$ are generators of the multiplication of $L$, and every element can be obtained by iterated Lie brackets of $e_1$ and $e_{m+1}$. We assume
\[
[e_1,e_{m+1}] = e_{m+2}, \quad [e_{m+1},e_{m+1}] = e_2.
\]
We take $B$ to be a homogeneous basis adapted to the lower
central series of $L$, with
\[
L^1=\operatorname{span}_{\mathbb F}\{e_2,e_{m+2}\}\oplus L^2.
\]
Suppose further that there exists an odd element
$x\in L^{p-2}$ , such that $[e_1,x]\ne 0$ or $[e_{m+1},x]\ne 0$.

If $[e_1,x]\ne 0$, let $[e_1,x]=y\in L^{p-1}$. Define an even linear map $T$ by
\[
T(e_{m+1}) = x, \quad T(e_{m+2}) =-y, \quad T(e_2) = -2[x,e_{m+1}], \quad T(e_i)=0,
\]
for all $i\ne 2,m+1,m+2$. 

By computation, T is an even anti-superderivation. It is clear that $T(L^1)\subseteq Z(L)$ and $ T| _{L^1}\neq 0 $.

If $[e_{m+1},x]\ne 0$, let $[x,e_{m+1}]=y\in L^{p-1}$. Define an even linear map $T$ by
\[
T(e_{m+1}) = x, \quad T(e_2) = -2y, \quad T(e_{m+2}) =-[e_1,x], \quad T(e_i)=0,
\]
for all $i\ne 2,m+1,m+2$. We have $T(L^1)\subseteq Z(L)$, $T| _{L^1}\neq 0$, and $T$ is an even anti-superderivation. By Theorem 4.5, $L$ admits pure local anti-superderivations.
\end{example}

\subsection*{Acknowledgements}
The authors are supported by the  NSF of Hei Longjiang Province (No. LH2024A014).
\subsection*{Declaration of competing interest}
There is no competing interest.


\end{document}